\documentclass[10pt]{article}
\usepackage{amssymb,latexsym,amsmath,epsfig,amsthm} 

\makeatletter

\renewcommand\section{\@startsection {section}{1}{\z@}
{-30pt \@plus -1ex \@minus -.2ex}
{2.3ex \@plus.2ex}
{\normalfont\normalsize\bfseries\boldmath}}

\renewcommand\subsection{\@startsection{subsection}{2}{\z@}
{-3.25ex\@plus -1ex \@minus -.2ex}
{1.5ex \@plus .2ex}
{\normalfont\normalsize\bfseries\boldmath}}

\renewcommand{\@seccntformat}[1]{\csname the#1\endcsname. }

\makeatother

\newtheorem{theorem}{Theorem}
\newtheorem{lemma}{Lemma}
\newtheorem{proposition}{Proposition}

\newtheorem*{problem}{Problem}
\theoremstyle{definition}
\newtheorem{definition}{Definition}

\begin{document}
																												
\begin{center}
\uppercase{\bf \boldmath On a Generalization of Zumkeller Numbers: Infinitude of odd $k$-IGMO numbers for all non-negative integers $k$}
\vskip 20pt
{\bf Li Ting Hon Stanford}\\
\end{center}

\centerline{\bf Abstract}
\noindent
In this paper, we introduce and investigate a novel generalization of Zumkeller numbers termed $k$-IGMO numbers, inspired by a problem proposed in the International Gamma Mathematical Olympiad (IGMO) 2025. A positive integer $n$ is defined as a $k$-IGMO number if its set of positive divisors can be partitioned into two disjoint subsets whose elements have sums differing by $k$. Under this definition, classical Zumkeller numbers correspond to the case where $k = 0$. The main result of the paper is the existence of infinitely many odd $k$-IGMO numbers for all non-negative integers $k$.

\pagestyle{myheadings}
\thispagestyle{empty}
\baselineskip=12.875pt
\vskip 30pt

	\section{Introduction}

    IGMO numbers were first described in a problem in International Gamma Mathematical Olympiad (IGMO) 2025. IGMO was a free to enter online Mathematics Olympiad organized by Global and Multicultural Mathematical Association (GAMMA), an association formed by mathematical educators, popularizers and amateur researchers \cite{IGMO}.  The problem was proposed by @Pepemaths from Instagram \cite{IGMO_numbers}, \cite{Pepemaths}. An IGMO number is defined as a positive integer whose set of positive divisors can be partitioned into two disjoint sets, such that the sum of the elements of the two sets have a difference of $5$. For example, $9$ is an IGMO number because we can partition the positive divisors of $9$ into two disjoint sets \{$1$, $3$\} and \{$9$\}, and the sum of elements of the two sets are $1+3=4$ and $9$ respectively, which have a difference of $5$.

    We generalize the concept of IGMO numbers and define a $k$-IGMO number, where $k$ is a non-negative integer, in the following way.

\begin{definition}\label{def}
  A positive integer $n$ is called a $k$-IGMO number if the set of positive divisors of $n$ can be partitioned as two disjoint sets $A$ and $B$ ($A \cup B = D_n$), such that 

  \[ \sum_{d\in A} d - \sum_{d\in B} d = k. \]
\end{definition}

    IGMO numbers originally defined in the problem in IGMO 2025 are equivalent to $5$-IGMO numbers in Definition \ref{def}. Note that $k$-IGMO numbers can also be considered as a generalization of Zumkeller numbers. A positive integer $n$ is called a Zumkeller number if the set of all positive divisors of $n$ can be partitioned into two disjoint subsets of equal sum \cite{Zumkeller_numbers}. The concept of Zumkeller numbers was first introduced in 1987 by LeVan \cite{integer_perfect}. Zumkeller published this class of numbers on the On-Line Encyclopedia of Integer Sequences (OEIS) in 2003 (sequence number A083207) \cite{OEIS} and popularized the concept. Zumkeller numbers are equivalent to a $0$-IGMO number in Definition \ref{def}.

    The aim of this article is to prove the existence of infinitely many odd $k$-IGMO numbers for all non-negative integers $k$.

    We use the following notations through the article. $\sigma(n)$ denotes the sum of positive divisors of a positive integer $n$. $D_n$ denotes the set of all positive divisors of a positive integer $n$.

 In Section \ref{sec: basic_properties}, we present and prove some basic properties of $k$-IGMO numbers.
 
    In Section \ref{sec: odd}, we prove the existence of infinitely many odd $k$-IGMO numbers for all odd positive integers $k$.
	
     In Section \ref{sec: even}, we prove the existence of infinitely many odd $k$-IGMO numbers for all even non-negative integers $k$. 
     
     In Section \ref{sec: summary}, we summarize our finding and give a concluding remark. 

	\section{Some basic properties of $k$-IGMO numbers}\label{sec: basic_properties}

    In this section, we present and prove some basic properties of $k$-IGMO numbers.

    Firstly, we give the following definition, which will be useful for our later discussion.

    \begin{definition}\label{def2}
  A non-empty set $A$ is called $(n,k)$-IGMO-able if all the elements of $A$ are positive divisors of a positive integer $n$ and

  \[ \sum_{d\in A} d  = \frac{\sigma(n)+k}{2}. \]
\end{definition}

We then present the following lemma which will be used throughout the article.

	\begin{lemma}\label{lem: existence_lemma}
If there exists a $(n,k)$-IGMO-able set, then $n$ is a $k$-IGMO number.
	\end{lemma}
	\begin{proof}
Suppose there exists a $(n,k)$-IGMO-able set $A$. Let $B= D_n \setminus A$, then $\sum_{d\in B} d  = \sigma(n) - \frac{\sigma(n)+k}{2} = \frac{\sigma(n)-k}{2}$. $\sum_{d\in A} d - \sum_{d\in B} d = \frac{\sigma(n)+k}{2} - \frac{\sigma(n)-k}{2}=k$. So $n$ is a $k$-IGMO number.
	\end{proof}

We present three basic properties of $k$-IGMO numbers.

	\begin{proposition}\label{proposition_1}
If $n$ is a $k$-IGMO number, then $k$ and $\sigma(n)$ have the same parity.
	\end{proposition}
	\begin{proof}
Let $A$ and $B$ be two disjoint subsets of the set $D_n$ of positive divisors of $n$ such that $A \cup B = D_n$ and $\sum_{d\in A} d - \sum_{d\in B} d = k$, then $\sigma(n) = \sum_{d\in A} d + \sum_{d\in B} d = (\sum_{d\in B} d + k) + \sum_{d\in B} d = 2\sum_{d\in B} d + k$. So $k$ and $\sigma(n)$ have the same parity.
	\end{proof}

  	\begin{proposition}\label{proposition_2}
If $n$ is an odd $k$-IGMO number and $k$ is odd, then $n$ is a perfect square.
	\end{proposition}
	\begin{proof}
Let the prime factorization of $n$ be $p_1^{\alpha_1}p_2^{\alpha_2} \dots p_m^{\alpha_m}$, where $p_1$, $p_2$, \dots, $p_m$ are odd primes, then $\sigma (n)=(1+p_1+p_1^2+ \dots +p_1^{\alpha_1})(1+p_2+p_2^2+ \dots +p_2^{\alpha_2}) \dots (1+p_m+p_m^2+ \dots +p_m^{\alpha_m})$, which is odd by Proposition \ref{proposition_1}. Since $\sigma(n)$ is odd, each individual factor $1 + p_i + p_i^2 + \dots + p_i^{\alpha_i}$ for $i=1,2,\dots,m$ must be odd. Because each $p_i$ is an odd prime, every power of $p_i$ is also odd. The sum $1 + p_i + p_i^2 + \dots + p_i^{\alpha_i}$ contains $\alpha_i + 1$ odd terms. For the total sum of these terms to be odd, the number of terms $\alpha_i + 1$ must be odd, which implies that each exponent $\alpha_i$ must be even. Thus, $n$ is a perfect square.
	\end{proof}

      	\begin{proposition}\label{proposition_3}
There exist infinitely many even $k$-IGMO numbers for all positive integers k.
	\end{proposition}
	\begin{proof}
Let $t$ be a positive integer such that $2^t > k$. The positive divisors of $2^t$ are $1$, $2$, $4$, $\dots$ , $2^t$. $\sigma(2^t)=2^{t+1}-1$, which is odd.

\vskip 5pt\noindent {\tt Case 1:} $k$ is odd.

$\frac{\sigma(2^t)+k}{2}=\frac{2^{t+1}-1+k}{2}$ is a positive integer which is $<2^{t+1}-1$. By considering the binary expression of $\frac{\sigma(2^t)+k}{2}$, we can express $\frac{\sigma(2^t)+k}{2}$ as the sum of some of the positive divisors of $2^t$, i.e. $1, 2, 4, \dots, 2^t$. The binary expansion uses only powers of 2 up to $2^t$  because $\frac{\sigma(2^t)+k}{2} < 2^{t+1}-1 = \sigma(2^t)$. Hence, there exists a $(2^t,k)$-IGMO-able set. By Lemma \ref{lem: existence_lemma}, $2^t$ is a $k$-IGMO number. There exist infinitely many $2^t > k$, so there exist infinitely many even $k$-IGMO numbers.

\vskip 5pt\noindent {\tt Case 2:} $k$ is even. 

The target sum $\frac{\sigma(3 \times 2^t)+k}{2} = 2(2^{t+1}-1)+\frac{k}{2}$. Since $2^t > k$, we have $\frac{k}{2} \leq 2^{t-1} - 1 < 2^{t+1}-1$. Thus, $\frac{\sigma(3 \times 2^t)+k}{2} < 3(2^{t+1}-1)$. 

By the division algorithm, we can write $\frac{\sigma(3 \times 2^t)+k}{2} = 3a + b$, where $b \in \{0, 1, 2\}$ and $a$ is an integer satisfying $a < 2^{t+1}-1$. Because $a < 2^{t+1}$, the binary expansion of $a$ uses only powers of 2 up to $2^t$. Let $a = 2^{\alpha_1} +2^{\alpha_2} +\dots + 2^{\alpha_s}$, where $2^{\alpha_1}, 2^{\alpha_2}, \dots 2^{\alpha_s}$ are distinct elements of $\{1, 2, 4, \dots, 2^t\}$. 

We construct the set $A$ as follows:
\begin{itemize}
    \item If $b = 0$, let $A = \{3 \times 2^{\alpha_1}, 3 \times 2^{\alpha_2}, \dots, 3 \times 2^{\alpha_s}\}$.
    \item If $b \in \{1, 2\}$, let $A = \{3 \times 2^{\alpha_1}, 3 \times 2^{\alpha_2}, \dots, 3 \times 2^{\alpha_s}, b\}$.
\end{itemize}
In either case, all elements in $A$ are distinct positive divisors of $3 \times 2^t$ (since $3 \times 2^{\alpha_i} \ge 3$, they never equal $1$ or $2$). The sum of elements of $A$ is $\frac{\sigma(3 \times 2^t)+k}{2}$, so $A$ is $(3 \times 2^t, k)$-IGMO-able. By Lemma \ref{lem: existence_lemma}, $3 \times 2^t$ is a $k$-IGMO number.

	\end{proof}

	\section{Existence of infinitely many odd $k$-IGMO numbers for all odd positive integers $k$}\label{sec: odd}

In this section, we prove the existence of infinitely many odd $k$-IGMO numbers, where $k$ is odd. In fact, if $k$ is odd, $1334025 \times 9^t$ is a $k$-IGMO number for all sufficiently large positive integers $t$.

We first present a few lemmas which are useful for later discussion.

    \begin{lemma}\label{lemma: condition}

For an odd positive integer $k$, $1334025 \times 9^t$ is $k$-IGMO for all $t \in \mathbb{Z}_{\ge 0}$ if there exists a set $A_k$ that is $(1334025,k)$-IGMO-able and a set $C_k$ which satisfies all the following conditions:

\begin{enumerate}

\item $C_k \subset D_{1334025 \times 9}$,

\item $\sum_{d\in C_k} d = 470022-4k$,

\item For all $d \in A_k$, $\nexists m \in \mathbb{Z}^+$ such that $9^md \in C_k$,

\item For all $d \in C_k$, $\nexists m \in \mathbb{Z}^+$ such that $9^md \in C_k$.

\end{enumerate}

    \end{lemma}

\begin{proof}
Suppose there exists a set $A_k$ that is $(1334025,k)$-IGMO-able and a set $C_k$ which satisfies all the aforementioned conditions. Let $A_{(1334025,k)}=A_k$. For $t \geq 1$, let $A_{(1334025 \times 9^t,k)}=\{9d \mid d \in A_{(1334025 \times 9^{t-1},k)}\} \cup C_k$. We shall prove by induction that $A_{(1334025 \times 9^t,k)}$ is $(1334025 \times 9^t,k)$-IGMO-able for all $t \in \mathbb{Z}_{\ge 0}$.

For $t=0$, $A_{(1334025,k)}=A_k$ is $(1334025,k)$-IGMO-able by definition. Assume that $A_{(1334025 \times 9^s,k)}$ is $(1334025 \times 9^s,k)$-IGMO-able for some $s \in \mathbb{Z}_{\ge 0}$. 

By our definition, any element $x \in \{9d \mid d \in A_{(1334025 \times 9^{s},k)}\}$ can be written in the form $x = 9^m d_0$ for some integer $m \ge 1$, where either $d_0 \in A_k$ or $d_0 \in C_k$. By Conditions 3 and 4, no element of this form can belong to $C_k$. Therefore, $\{9d \mid d \in A_{(1334025 \times 9^{s},k)}\}$ and $C_k$ are completely disjoint sets. Moreover, because $C_k \subset D_{1334025 \times 9}$ and $A_{(1334025 \times 9^{s},k)} \subset D_{1334025 \times 9^s}$, all elements of $A_{(1334025 \times 9^{s+1},k)}$ are positive divisors of $1334025 \times 9^{s+1}$.

Since $A_{(1334025 \times 9^s,k)}$ is $(1334025 \times 9^s,k)$-IGMO-able, we have $\sum_{d\in A_{(1334025 \times 9^s,k)}} d = \frac{\sigma(1334025 \times 9^s)+k}{2}$. We compute the sum of elements of $A_{(1334025 \times 9^{s+1},k)}$:
\begin{align*}
\sum_{d\in A_{(1334025 \times 9^{s+1},k)}} d &= 9 \bigg( \sum_{d\in A_{(1334025 \times 9^s,k)}} d \bigg) + \sum_{d\in C_k} d \\
&= 9 \bigg[ \frac{\sigma(1334025 \times 9^s)+k}{2} \bigg] + (470022-4k)\\
&= 9 \bigg[ \frac{6345297 \times 9^s - 235011 + 2k}{4} \bigg] + \frac{1880088-16k}{4}\\
&= \frac{6345297 \times 9^{s+1} - 2115099 + 18k + 1880088 - 16k}{4}\\
&= \frac{6345297 \times 9^{s+1} - 235011 + 2k}{4}\\
&= \frac{\frac{6345297 \times 9^{s+1} - 235011}{2}+k}{2}\\
&= \frac{\sigma(1334025 \times 9^{s+1})+k}{2}.
\end{align*}

Thus, $A_{(1334025 \times 9^{s+1},k)}$ is $(1334025 \times 9^{s+1},k)$-IGMO-able. By mathematical induction, $A_{(1334025 \times 9^t,k)}$ is $(1334025 \times 9^t,k)$-IGMO-able for all $t \in \mathbb{Z}_{\ge 0}$. Hence, $1334025 \times 9^t$ is a $k$-IGMO number for all $t \in \mathbb{Z}_{\ge 0}$.
\end{proof}

    \begin{lemma}\label{lemma: complete}

Let $S=\{1,3,7,11\} \ \cup \{3^{t+1} \times 11,3^t \times 5 \times 11,3^t \times 7 \times 11\ \mid t \in \mathbb{Z}_{\ge 0} \}$. All positive multiples of $11$ can be represented as a sum of distinct elements of $S$. 

    \end{lemma}

\begin{proof}
Consider the sequence $\{v_i\}$ where $v_1=v_2=1$, and for $i \geq 3$:
$$v_i=\begin{cases}
  3^{\frac{i}{3}}, & \text{if $i \equiv 0 \pmod{3}$},\\
  5 \times 3^{\frac{i-4}{3}}, & \text{if $i \equiv 1 \pmod{3}$},\\
  7 \times 3^{\frac{i-5}{3}}, & \text{if $i \equiv 2 \pmod{3}$}.
\end{cases}$$

Note that the sequence is non-decreasing. According to Brown's criterion \cite{brown}, a non-decreasing sequence of positive integers $\{u_i\}$ with $u_1=1$ is complete, i.e. every positive integer can be expressed as sum of distinct terms of the sequence, if and only if for all $i \geq 1$:
$$\sum_{j=1}^{i} u_j \geq u_{i+1} - 1.$$

It is easy to verify the cases for $1 \le i \le 5$. For $i \geq 6$,
$$\sum_{j=1}^{i} v_j=\begin{cases}
  \frac{7 \times 3^{\frac{i}{3}}-11}{2}, & \text{if $i \equiv 0 \pmod{3}$},\\
  \frac{31 \times 3^{\frac{i-4}{3}}-11}{2}, & \text{if $i \equiv 1 \pmod{3}$},\\
  \frac{15 \times 3^{\frac{i-2}{3}}-11}{2}, & \text{if $i \equiv 2 \pmod{3}$}.
\end{cases}$$

Now we check Brown's inequality $\sum_{j=1}^{i} v_j \geq v_{i+1} - 1$ for $i \geq 6$. The inequalities can be simplified to:

  $$\begin{cases}
  11 \times 3^{\frac{i}{3}-1} \geq 9, & \text{if $i \equiv 0 \pmod{3}$},\\
  17 \times 3^{\frac{i-4}{3}} \geq 9, & \text{if $i \equiv 1 \pmod{3}$},\\
  9 \times 3^{\frac{i-2}{3}} \geq 9, & \text{if $i \equiv 2 \pmod{3}$}.
  \end{cases}$$
  
These inequalities are true for $i \geq 6$. Thus, $\{v_i\}$ satisfies Brown's criterion and is a complete sequence. Every positive integer $k \in \mathbb{Z}^+$ can be expressed as a sum of a subsequence of distinct elements of $\{v_i\}$.

Let $11k$ be a positive multiple of $11$. We can express $k$ as $v_{i_1} + v_{i_2} + \dots + v_{i_j}$, for some $i_1 < i_2 < \dots < i_j$. So $11k = 11v_{i_1} + 11v_{i_2} + \dots + 11v_{i_j}$.

If $\{1, 2\} \subseteq \{i_1, i_2, \dots, i_j\}$, then the sum contains $11v_1 + 11v_2 = 11(1) + 11(1) = 22$. We substitute this pair with $1 + 3 + 7 + 11$. The remaining terms in the sum are of the form $11v_m$ for $m \geq 3$. By definition, for $m \geq 3$, each $11v_m$ is a distinct element in $\{3^{t+1} \times 11, \; 3^t \times 5 \times 11, \; 3^t \times 7 \times 11 \mid t \in \mathbb{Z}_{\ge 0} \}$, all of which are strictly greater than $11$. Thus, all elements in the final sum are distinct members of $S$.

If $\{1, 2\} \not\subseteq \{i_1, i_2, \dots, i_j\}$, then the set $\{i_1, i_2, \dots, i_j\}$ contains at most one element from $\{1, 2\}$. If it contains exactly one element from $\{1, 2\}$, it contributes a single term equal to $11$. All other terms are $11v_m$ for $m \geq 3$, which are distinct from each other and strictly greater than $11$. If it contains neither $1$ nor $2$, all terms are $11v_m$ for $m \geq 3$. In both cases, $11k$ is expressed as a sum of distinct elements of $S$.
Hence, any positive multiple of $11$ can be represented as a sum of distinct elements of $S$.
\end{proof}

        \begin{lemma}\label{lemma: k-IGMO+22}

For an odd positive integer $k=x+22y$, where $x$ is a positive odd integer $<22$ and $y$ is a positive integer, suppose there exist set $A_x$ and $C_x$ which satisfy the conditions:

\begin{enumerate}

\item $A_x$ is $(1334025,x)$-IGMO-able,

\item $C_x \subset D_{1334025 \times 9}$,

\item $\sum_{d\in C_x} d = 470022-4x$,

\item For all $d \in A_x$, $\nexists m \in \mathbb{Z}^+$ such that $9^md \in C_x$,

\item For all $d \in C_x$, $\nexists m \in \mathbb{Z}^+$ such that $9^md \in C_x$.

\item Let $S=\{1,3,7,11\} \ \cup \{3^{t+1} \times 11,3^t \times 5 \times 11,3^t \times 7 \times 11\ \mid t \in \mathbb{Z}_{\ge 0} \}$.  For all $d \in C_x$, $\nexists m \in \mathbb{Z}_{\ge 0}$ such that $9^md \in S$.

\end{enumerate}

Then there exist infinitely many odd $k$-IGMO numbers.
    \end{lemma}

\begin{proof}
Suppose all the conditions are satisfied. Let $A_{(1334025,x)}=A_x$. For $t \ge 1$, let $A_{(1334025 \times 9^t,x)}=\{9d \mid d \in A_{(1334025 \times 9^{t-1},x)}\} \cup C_x$. By Lemma \ref{lemma: condition}, $A_{(1334025 \times 9^t,x)}$ is $(1334025 \times 9^t,x)$-IGMO-able for all $t \in \mathbb{Z}_{\ge 0}$.

By Lemma \ref{lemma: complete}, $11y$, which is a positive multiple of $11$, can be represented as a sum of distinct elements of $S$, so $11y = u_1+u_2 +\dots +u_i$ for distinct $u_1,u_2 ,\dots ,u_i \in S$. There exist a sufficiently large integer $t_1$ such that if $t > t_1$, then $\{u_1,u_2 ,\dots ,u_i\} \subset D_{1334025 \times 9^t}$. 

Furthermore, note that $A_{(1334025 \times 9^t,x)} = (9^tA_x) \cup C_x \cup (9C_x) \cup (9^2C_x) \cup \dots \cup (9^{t-1}C_x)$. There exists a sufficiently large integer $t_2$ such that if $t > t_2$, all the elements of $9^tA_x$ are larger than $\max(u_1,u_2 ,\dots ,u_i)$, and so $\{u_1,u_2 ,\dots ,u_i\} \cap 9^tA_x = \emptyset$. By Condition $6$, no element of $S$ can be generated by multiplying an element of $C_x$ by a non-negative power of $9$, so $\{u_1,u_2 ,\dots ,u_i\} \cap [C_x \cup (9C_x) \cup (9^2C_x) \cup \dots \cup (9^{t-1}C_x)] = \emptyset$. Overall, for all $t > t_2$, $\{u_1,u_2 ,\dots ,u_i\} \cap A_{(1334025 \times 9^t,x)} = \emptyset$.

For all $t > \max(t_1,t_2)$, let $A'_{(1334025 \times 9^t,x)} = A_{(1334025 \times 9^t,x)} \cup \{u_1,u_2 ,\dots ,u_i\}$.
\begin{align*}
\sum_{d\in A'_{(1334025 \times 9^t,x)}} d &= \sum_{d\in A_{(1334025 \times 9^t,x)}} d + \sum_{m=1}^i u_m \\
&= \frac{\sigma(1334025 \times 9^t)+x}{2} + 11y \\
&= \frac{\sigma(1334025 \times 9^t)+(x+22y)}{2} \\
&= \frac{\sigma(1334025 \times 9^t)+k}{2}.
\end{align*}

Hence, by Definition \ref{def2}, $A'_{(1334025 \times 9^t,x)}$ is $(1334025 \times 9^t,k)$-IGMO-able. This implies that for every $t > \max(t_1,t_2)$, the integer $1334025 \times 9^t$ is an odd $k$-IGMO number. Because there are infinitely many such values of $t$, there exist infinitely many odd $k$-IGMO numbers.
\end{proof}

We shall use these lemmas to prove the final result of this section.
    
    \begin{lemma}\label{lemma: IGMOodd}

There exist infinitely many odd $k$-IGMO numbers for all odd positive integers $k$.
    \end{lemma}
    
	\begin{proof}

Refer to the following table, it is easy to verify that the following sets satisfy all the conditions of Lemma \ref{lemma: condition} and Lemma \ref{lemma: k-IGMO+22}. 

\begin{center}
\begin{tabular}{|c|c|c|}
\hline
$k$ or $x$ & $A_k$ or $A_x$ & $C_k$ or $C_x$ \\ \hline
$1$  &   $\{7,25,245,2695,190575,1334025\}$   &  $\{5,35,49,75,385,1925,22869,444675\}$    \\ \hline
$3$ &  $\{33,245,2695,190575,1334025\}$    &    $\{441,2025,22869,444675\}$  \\ \hline
$5$ &   $\{1,3,275,2695,190575,1334025\}$       &        $\{15,49,189,2205,22869,444675\}$  \\ \hline
$7$ &   $\{5,275,2695,190575,1334025\}$   &   $\{385,525,5929,8085,10395,444675\}$   \\ \hline
$9$ &   $\{1,5,275,2695,190575,1334025\}$    &    $\{825,1617,22869,444675\}$  \\ \hline
$11$  &   $\{7,275,2695,190575,1334025\}$    &  $\{5,35,189,2205,22869,444675\}$    \\ \hline
$13$ &   $\{1,7,275,2695,190575,1334025\}$   &    $\{21,25,175,2205,22869,444675\}$  \\ \hline
$15$ &  $\{9,275,2695,190575,1334025\}$    &   $\{21,45,147,2205,22869,444675\}$   \\ \hline
$17$  &   $\{1,9,275,2695,190575,1334025\}$   &    $\{5,25,175,2205,22869,444675\}$  \\ \hline
$19$ &  $\{11,275,2695,190575,1334025\}$    &    $\{15,35,147,2205,22869,444675\}$  \\ \hline
$21$  &  $\{1,11,275,2695,190575,1334025\}$    &   $\{189,2205,22869,444675\}$   \\ \hline
\end{tabular}
\end{center}

By Lemma \ref{lemma: condition}, for an odd positive integer $k < 22$, $1334025 \times 9^t$ is $k$-IGMO for all $t \in \mathbb{Z}_{\ge 0}$. Hence, there exist infinitely many odd $k$-IGMO numbers for all odd positive integers $k < 22$.

For a positive odd integer $k>22$, it can be expressed in the form of $k=x+22y$, where $x$ is a positive odd integer $<22$ and $y$ is a positive integer. By Lemma \ref{lemma: k-IGMO+22}, there exist infinitely many odd $k$-IGMO numbers.

Overall, there exist infinitely many odd $k$-IGMO numbers for all odd positive integers $k$.

        \end{proof}
        
    \section{Existence of infinitely many odd $k$-IGMO numbers for all even non-negative integers $k$}\label{sec: even}
        
	In this section, we prove the existence of infinitely many odd $k$-IGMO numbers for all even non-negative integers $k$. We require some lemmas related to Zumkeller numbers to prove our results. We give the following definition which will be useful for later discussion.

    \begin{definition}\label{IGMO-helpful}
  A positive integer $n$ is called a IGMO-helpful Zumkeller number if the set of positive divisors of $n$ can be partitioned as two disjoint sets $A$ and $B$, such that $\{1,3,5,7,9\}\subseteq A$, $15 \in B$ and $\sum_{d\in A} d = \sum_{d\in B} d$.
\end{definition}

        	\begin{lemma}\label{lem: zumkeller2}
        
If $n$ is a IGMO-helpful Zumkeller number and $w$ is a positive integer that is relatively prime to $n$, then $nw$ is also a IGMO-helpful Zumkeller number.

	\end{lemma}
	\begin{proof}
		Let $\{A,B\}$ be a IGMO-helpful Zumkeller partition of $n$. Let $1,d_1,d_2,\dots,d_i$ be all the positive divisors of $w$. $n$ and $w$ are relatively prime to each other. This ensures that 
        $A$, $(d_1A),(d_2A),\dots,(d_iA),B,(d_1B),(d_2B),\dots,(d_iB)$ are disjoint sets. Then $\{A \cup (d_1A) \cup (d_2A) \cup \dots \cup (d_iA),B\cup(d_1B)\cup(d_2B)\cup \dots \cup (d_iB)\}$ is a IGMO-helpful Zumkeller partition of $nw$. Hence, $nw$ is a IGMO-helpful Zumkeller number.
        
	\end{proof}

           	\begin{lemma}\label{lem: zumkeller3}
        
If $n$ is a IGMO-helpful Zumkeller number and $p_1^{k_1}p_2^{k_2} \dots p_m^{k_m}$ is the prime factorization of $n$. Then for any non-negative integers $l_1$, $l_2$, $\dots$, $l_m$, the integer $p_1^{k_1+l_1(k_1+1)}p_2^{k_2+l_2(k_2+1)} \dots p_m^{k_m+l_m(k_m+1)}$ is also a IGMO-helpful Zumkeller number.

	\end{lemma}
	\begin{proof}
		The proof is modified from \cite[Proposition 6]{Zumkeller_numbers}. It is sufficient to show that $p_1^{k_1+l_1(k_1+1)}p_2^{k_2} \dots p_m^{k_m}$ is also a IGMO-helpful Zumkeller number, as the full result follows by iteratively applying the identical steps to each remaining prime factor $p_2, \dots, p_m$. Let $\{A,B\}$ be a IGMO-helpful Zumkeller partition of $n$. Let $D_n$ be the set of positive divisors of $n$, then the set of positive divisors of $p_1^{k_1+l_1(k_1+1)}p_2^{k_2} \dots p_m^{k_m}$ is $D_n \cup (p_1^{k_1+1}D_n) \cup (p_1^{2(k_1+1)}D_n) \cup \dots \cup (p_1^{l_1(k_1+1)}D_n)$. So $\{A \cup (p_1^{k_1+1}A) \cup (p_1^{2(k_1+1)}A) \cup \dots \cup (p_1^{l_1(k_1+1)}A) , B \cup (p_1^{k_1+1}B) \cup (p_1^{2(k_1+1)}B) \cup \dots \cup (p_1^{l_1(k_1+1)}B)\}$ is a IGMO-helpful Zumkeller partition of $p_1^{k_1+l_1(k_1+1)}p_2^{k_2} \dots p_m^{k_m}$.
                
	\end{proof}

   	\begin{lemma}\label{lem: zumkeller6}

 There exist infinitely many odd IGMO-helpful Zumkeller numbers.

	\end{lemma}

\begin{proof}
    $945$ is an odd IGMO-helpful Zumkeller number because the set of its positive divisors can be partitioned into disjoint sets $A= \{1,3,5,7,9,21,27,35,45,63,105,135,189,315\}$ and $B= \{15,945\}$, which satisfies the condition of IGMO-helpful Zumkeller. By Lemma \ref{lem: zumkeller2}, if $w$ is an odd positive integer that is relatively prime to $945$, then $945w$ is also an odd IGMO-helpful Zumkeller number. There exist infinitely many odd positive integers $w$ which are relatively prime to $945$, so there exist infinitely many odd IGMO-helpful Zumkeller numbers.
\end{proof}

                	\begin{lemma}\label{lem: zumkeller7}
        
 For all odd positive integers $k$, there exist infinitely many odd IGMO-helpful Zumkeller numbers which are multiples of $k$.

	\end{lemma}

\begin{proof}
Let $k=3^{\alpha}5^{\beta}7^{\gamma}s$, where $\alpha, \beta, \gamma$ are non-negative integers, and $s$ is an odd positive integer that is relatively prime to $3, 5,$ and $7$. 

Note that $945 = 3^3 \times 5 \times 7$ is an odd IGMO-helpful Zumkeller number. There exist infinitely many combinations of non-negative integers $l_1, l_2, l_3$ such that $3+4l_1 \geq \alpha$, $1+2l_2 \geq \beta$, and $1+2l_3 \geq \gamma$. By Lemma \ref{lem: zumkeller3}, $3^{3+4l_1}5^{1+2l_2}7^{1+2l_3}$ is also an odd IGMO-helpful Zumkeller number. 

Since $3^{3+4l_1}5^{1+2l_2}7^{1+2l_3}$ is composed strictly of powers of $3$, $5$ and $7$, it is relatively prime to $s$. By Lemma \ref{lem: zumkeller2}, multiplying this number by $s$ yields a new odd IGMO-helpful Zumkeller number $3^{3+4l_1}5^{1+2l_2}7^{1+2l_3}s$. Because $3+4l_1 \geq \alpha$, $1+2l_2 \geq \beta$, $1+2l_3 \geq \gamma$, and the expression includes $s$ as a factor, $3^{3+4l_1}5^{1+2l_2}7^{1+2l_3}s$ is a multiple of $k$. 

Since there are infinitely many valid combinations of $l_1, l_2,$ and $l_3$, there exist infinitely many odd IGMO-helpful Zumkeller numbers which are multiples of $k$.
\end{proof}

We shall use these lemmas to prove the final result of this section.

       \begin{lemma}\label{lemma: evenIGMO}

There exist infinitely many odd $k$-IGMO numbers for all even non-negative integers $k$.
    \end{lemma}
    
   \begin{proof}
\vskip 5pt\noindent {\tt Case 1:} $k=0$

By Lemma \ref{lem: zumkeller6}, there exist infinitely many odd IGMO-helpful Zumkeller numbers. Since a $0$-IGMO number is by definition equivalent to a Zumkeller number, these are all $0$-IGMO numbers.

\vskip 5pt\noindent {\tt Case 2:} $k=8$

By Lemma \ref{lem: zumkeller6}, there exist infinitely many odd IGMO-helpful Zumkeller numbers. Let $n$ be an odd IGMO-helpful Zumkeller number, and let $\{A,B\}$ be its corresponding partition. Define $A'=A \setminus \{1,3\}$ and $B'=B \cup \{1,3\}$. Then $A'$ and $B'$ form a valid disjoint partition of the divisors of $n$, and
$$ \sum_{d\in B'} d - \sum_{d\in A'} d = \left(\sum_{d\in B} d + 1 + 3\right) - \left(\sum_{d\in A} d - 1 - 3\right) = 8 .$$
Thus, $n$ is an $8$-IGMO number. Hence, there exist infinitely many odd $8$-IGMO numbers.

\vskip 5pt\noindent {\tt Case 3:} $k=16$

Let $A'=A \setminus \{3,5\}$ and $B'=B \cup \{3,5\}$. Then $A'$ and $B'$ are disjoint subsets partitioning the divisors of $n$, and
$$ \sum_{d\in B'} d - \sum_{d\in A'} d = \left(\sum_{d\in B} d + 3 + 5\right) - \left(\sum_{d\in A} d - 3 - 5\right) = 16 .$$
Thus, $n$ is a $16$-IGMO number. Hence, there exist infinitely many odd $16$-IGMO numbers.

\vskip 5pt\noindent {\tt Case 4:} $k=28$

Let $A'=A \setminus \{5,9\}$ and $B'=B \cup \{5,9\}$. Then $A'$ and $B'$ are disjoint subsets partitioning the divisors of $n$, and
$$ \sum_{d\in B'} d - \sum_{d\in A'} d = \left(\sum_{d\in B} d + 5 + 9\right) - \left(\sum_{d\in A} d - 5 - 9\right) = 28 .$$
Thus, $n$ is a $28$-IGMO number. Hence, there exist infinitely many odd $28$-IGMO numbers.

\vskip 5pt\noindent {\tt Case 5:} $k$ is in the form $4m+2$, where $m \in \mathbb{Z}_{\ge 0}$

By Lemma \ref{lem: zumkeller7}, there exist infinitely many odd IGMO-helpful Zumkeller numbers which are multiples of $\frac{k}{2}$ (which is an odd integer of the form $2m+1$). Let $n$ be such a multiple with Zumkeller partition $\{A,B\}$. Without loss of generality, assume $\frac{k}{2} \in B$. Let $A'=A \cup \{\frac{k}{2}\}$ and $B' = B \setminus \{\frac{k}{2}\}$. Then $A'$ and $B'$ are disjoint partitions of the divisors of $n$, and
$$ \sum_{d\in A'} d - \sum_{d\in B'} d = \left(\sum_{d\in A} d + \frac{k}{2}\right) - \left(\sum_{d\in B} d - \frac{k}{2}\right) = k. $$
Thus, $n$ is a $k$-IGMO number. Hence, there exist infinitely many odd $k$-IGMO numbers, where $k$ is in the form $4m+2$ and $m \in \mathbb{Z}_{\ge 0}$.

\vskip 5pt\noindent {\tt Case 6:} $k$ is in the form $4m$, where $m \in \mathbb{Z}^+ \setminus \{2,4,7\}$

By Lemma \ref{lem: zumkeller7}, there exist infinitely many odd IGMO-helpful Zumkeller numbers which are multiples of $\frac{k}{2}+1$, which is an odd integer. Let $n$ be such a multiple with IGMO-helpful Zumkeller partition $\{A,B\}$. Note that $\frac{k}{2}+1 \notin \{1,5,9,15\}$ due to the restrictions on $m$. Since $\frac{k}{2}+1$ is a positive divisor of $n$, it must belong to either $A$ or $B$.

If $\frac{k}{2}+1 \in A$, let $A'=(A \setminus \{\frac{k}{2}+1,5,9\}) \cup \{15\}$ and $B'=(B \setminus \{15\}) \cup \{\frac{k}{2}+1,5,9\}$. These sets are disjoint partitions of the divisors of $n$, and
$$ \sum_{d\in B'} d - \sum_{d\in A'} d = \left[\sum_{d\in B} d + \left(\frac{k}{2}+1\right)+5+9-15\right] - \left[\sum_{d\in A} d - \left(\frac{k}{2}+1\right)-5-9+15\right] = k .$$

If $\frac{k}{2}+1 \in B$, let $A'=(A \setminus \{1\}) \cup \{\frac{k}{2}+1\}$ and $B'=(B \setminus \{\frac{k}{2}+1\}) \cup \{1\}$. These sets are disjoint partitions of the divisors of $n$, and
$$ \sum_{d\in A'} d - \sum_{d\in B'} d = \left[\sum_{d\in A} d + \left(\frac{k}{2}+1\right)-1\right] - \left[\sum_{d\in B} d - \left(\frac{k}{2}+1\right)+1\right] = k .$$
In both cases, $n$ is a $k$-IGMO number. Hence, there exist infinitely many odd $k$-IGMO numbers, where $k$ is in the form $4m$ and $m \in \mathbb{Z}^+ \setminus \{2,4,7\}$.

\bigskip

Combining the cases, we prove that there exist infinitely many odd $k$-IGMO numbers for all even non-negative integers $k$.
\end{proof}

	\section{Concluding Remarks}\label{sec: summary}

 In this article, we have generalized the concept of IGMO numbers and systematically classified their existential properties. Our investigations yield the following main theorem.

\begin{theorem}\label{final}
There exist infinitely many odd $k$-IGMO numbers for all non-negative integers $k$.
\end{theorem}

\begin{proof}
The result follows immediately from the combination of Lemma \ref{lemma: IGMOodd} (for all odd positive integers $k$) and Lemma \ref{lemma: evenIGMO} (for all even non-negative integers $k$).
\end{proof}

We conclude this article by further generalizing the concept of $k$-IGMO numbers and propose the following open problem, which offers a compelling direction for future research.

\begin{problem}\label{multi-partite}
Let $n$ be a positive integer, and let $D_n$ denote its set of positive divisors. We define $n$ to be a $(k_1, k_2, \dots, k_{m-1})$-IGMO number if $D_n$ can be partitioned into $m$ mutually disjoint subsets $A_1, A_2, \dots, A_m$ such that $\bigcup_{i=1}^m A_i = D_n$, and for all $1 \leq i \leq m-1$,
\[ \sum_{d\in A_{i+1}} d - \sum_{d\in A_i} d = k_i. \]
Characterize the combinations of $(k_1, k_2, \dots, k_{m-1})$ such that there exist infinitely many odd or even $(k_1, k_2, \dots, k_{m-1})$-IGMO numbers.

\end{problem}

\vskip20pt\noindent {\bf Acknowledgements.} The author thanks Sai Teja Somu and Andrzej Kukla for reviewing the manuscript and providing insightful suggestions.


\begin{thebibliography}{1}\footnotesize

\bibitem{IGMO} Wassupevery1, 4th IGMO is here!, Arts of Problem Solving Forums. (accessed April 12, 2026) {\tt https://artofproblemsolving.com/community/q1h3501899p34024733}.

\bibitem{IGMO_numbers} Wassupevery1, IGMO numbers, Arts of Problem Solving Forums. (accessed April 12, 2026) {\tt https://artofproblemsolving.com/community/q1h3514527p34112137}.

\bibitem{Pepemaths} Pepemaths, Thank you for supporting IGMO!!!, Instagram. (accessed April 12, 2026) {\tt https://www.instagram.com/pepemaths/DGioaV0JnEq}.

\bibitem{Zumkeller_numbers} Y. Peng and K.P.S. Bhaskara Rao, On Zumkeller numbers, {\it Journal of Number Theory} {113} (4) (2013), 1135--1155.

\bibitem{integer_perfect} M. O. LeVan, Integer-perfect numbers, {\it Journal of Natural Sciences and Mathematics} {27} (2) (1987), 33--50.

\bibitem{OEIS} OEIS Foundation Inc.,  The On-Line Encyclopedia of Integer Sequences, {\tt https://oeis.org}.

\bibitem{brown} J. L. Jr. Brown, Notes on Complete Sequences of Integers, {\it Amer. Math. Monthly} {68} (1961), 557--560.

\end{thebibliography}
\end{document}